\documentclass{article}

\usepackage{amssymb}
\usepackage{amsmath}
\usepackage{amsthm}
\usepackage{thmtools}

\usepackage{marginnote}
\usepackage{enclose}
\usepackage{anyfontsize}
\usepackage{color}
\usepackage{comment}
\usepackage{graphicx}
\usepackage{enumitem}
\usepackage{hyperref}
\usepackage{xcolor}
\hypersetup{
    colorlinks,
    linkcolor={blue!50!black},
    citecolor={green!50!black},
    urlcolor={red!80!black}
}
\DeclareMathAlphabet{\mathpzc}{OT1}{pzc}{m}{it}

\usepackage[top=2cm,bottom=2cm]{geometry}
\usepackage{calc}

\usepackage[normalem]{ulem}

\providecommand{\empmeasure}{{\ensuremath{{\lambda}}}}
\providecommand{\probspace}{\mathbb P}
\providecommand{\driftA}{\mathpzc A}
\providecommand{\volB}{\mathpzc B}

\providecommand{\prob}[1]{{\ensuremath{\mathrm{P}}\mspace{-2mu}\left[#1\right]}}%
\providecommand{\E}[1]{{\ensuremath{\mathrm{E}}\mspace{-2mu}\en*[]{#1}}}%
\providecommand{\var}[1]{{\ensuremath{\mathrm{Var}}\mspace{-2mu}\left[#1\right]}}%
\providecommand{\cov}[1]{{\ensuremath{\mathrm{Cov}}\mspace{-2mu}\left[#1\right]}}%
\providecommand{\tol}{\mathrm{TOL}}
\providecommand{\rset}{\mathbb{R}}

\providecommand{\nset}{\mathbb{N}}

\providecommand{\Order}[1]{ {\ensuremath{ \mathcal O\left( #1
      \right)}} }  
\providecommand{\mykeywords}[1]{\par \textbf{Keywords:} #1}
\providecommand{\subclass}[1]{\par{ \bf Class:} #1}
\renewcommand{\vec}[1]{{\ensuremath{{\boldsymbol #1}}}}
\providecommand{\valpha}{{{\ensuremath{\vec{\alpha}}}}}

\def\Tiny{\fontsize{4pt}{4pt}\selectfont}
\newcommand*{\MyDef}{\textnormal{\Tiny def}}
\newcommand*{\eqdefU}{\ensuremath{\mathop{\overset{\MyDef}{=}}}}%
\newcommand*{\eqdef}{\hskip 0.03cm\mathop{\overset{\MyDef}{\resizebox{\widthof{\eqdefU}}{\heightof{=}}{=}}}\hskip 0.03cm}

\providecommand{\di}{{\ensuremath{ \mathrm{d} }}}

\providecommand{\sst}{{\ensuremath{s_{\textnormal{t}}}}}
\providecommand{\betat}{{\ensuremath{\beta_{\textnormal{t}}}}}

\providecommand{\ssp}{{\ensuremath{s_{\textnormal{p}}}}}
\providecommand{\ggp}{{\ensuremath{\gamma_{\textnormal{p}}}}}
\providecommand{\betap}{{\ensuremath{\beta_{\textnormal{p}}}}}

\providecommand{\work}[1]{\textnormal{Work}\left[#1\right]}

\allowdisplaybreaks[1] \title{Multilevel and Multi-index Monte Carlo
  methods for the McKean-Vlasov equation}

\author{
  Abdul-Lateef~Haji-Ali
  \and
  Ra\'{u}l~Tempone
}

\declaretheorem[numberwithin=section]{theorem}
\declaretheorem[name=Example]{example*}

\declaretheorem[sharenumber=theorem, name=Lemma]{lemma}

\declaretheorem[name=Kuramoto Example,numbered=no]{kuramoto}

\defenclose{\op}()%

\begin{document}
\pagestyle{myheadings} \thispagestyle{plain} \markboth{Monte Carlo
  methods for stochastic particle systems in the mean-field}{Monte
  Carlo methods for stochastic particle systems in the mean-field}

\maketitle

\begin{abstract}
  We address the approximation of functionals depending on a system of
  particles, described by stochastic differential equations (SDEs), in
  the mean-field limit when the number of particles approaches
  infinity. This problem is equivalent to estimating the weak solution
  of the limiting McKean-Vlasov SDE. To that end, our approach uses
  systems with finite numbers of particles and a time-stepping
  scheme. In this case, there are two discretization parameters: the
  number of time steps and the number of particles.
  Based on these two parameters, we consider different variants of the
  Monte Carlo and Multilevel Monte Carlo (MLMC) methods and show that,
  in the best case, the optimal work complexity of MLMC, to estimate
  the functional in one typical setting with an error tolerance of
  $\tol$, is $\Order{\tol^{-3}}$.  We also consider a method that uses
  the recent Multi-index Monte Carlo method and show an improved work
  complexity in the same typical setting of
  $\Order{\tol^{-2}\log(\tol^{-1})^2}$. Our numerical experiments are
  carried out on the so-called Kuramoto model, a system of coupled
  oscillators.  \mykeywords{Multi-index Monte Carlo, Multilevel Monte
    Carlo, Monte Carlo, Particle systems, McKean-Vlasov, Mean-field,
    Stochastic Differential Equations, Weak Approximation, Sparse
    Approximation, Combination technique} \subclass{65C05 (Monte Carlo
    methods), 65C30 (Stochastic differential and integral equations),
    65C35 (Stochastic particle methods)}
\end{abstract}

\section{Introduction}\label{s:intro}
In our setting, a stochastic particle system is a system of
\emph{coupled} $d$-dimensional stochastic differential equations
(SDEs), each modeling the state of a ``particle''. Such particle
systems are versatile tools that can be used to model the dynamics of
various complicated phenomena using relatively simple interactions,
e.g., pedestrian dynamics \cite{helbing1995social, hajiali:msthesis},
collective animal behavior \cite{erban2015cucker,
  erban2011individual}, interactions between cells
\cite{dobramysl2015particle} and in some numerical methods such as
Ensemble Kalman filters \cite{del2016stability}.
One
common goal of the simulation of these particle systems is to average
some quantity of interest computed on \emph{all} particles, e.g., the
average velocity, average exit time or average number of particles in
a specific region.

Under certain conditions, most importantly the exchangeability of
particles and sufficient regularity of the SDE coefficients, the
stochastic particle system approaches a mean-field limit as the number
of particles tends to infinity
\cite{sznitman1991topics}. Exchangeability of particles refers to the
assumption that all permutations of the particles have the same joint
distribution.
In the mean-field limit, each particle follows a single McKean-Vlasov
SDE where the advection and/or diffusion coefficients depend on the
distribution of the solution to the SDE \cite{gartner:mckean}.  In
many cases, the objective is to approximate the expected value of a
quantity of interest (QoI) in the mean-field limit as the number of
particles tend to infinity, subject to some error tolerance, $\tol$.
While it is possible to approximate the expectation of these QoIs by
estimating the solution to a nonlinear PDE using traditional numerical
methods, such methods usually suffer from the curse of
dimensionality. Indeed, the cost of these method is usually of
$\Order{\tol^{-\frac{w}{d}}}$ for some constant $w>1$ that depends on
the particular numerical method. Using sparse numerical methods
alleviates the curse of dimensionality but requires increasing
regularity as the dimensionality of the state space increases.  On the
other hand, Monte Carlo methods do not suffer from this curse with
respect to the dimensionality of the state space.  This work explores
different variants and extensions of the Monte Carlo method when the
underlying stochastic particle system satisfies certain crucial
assumptions.
We theoretically show the validity of some of these assumptions in a
somewhat general setting, while verifying the other assumptions
numerically on a simple stochastic particle system, leaving further
theoretical justification to a future work.

Generally, the SDEs that constitute a stochastic particle system
cannot be solved exactly and their solution must instead be
approximated using a time-stepping scheme with a number of time steps,
$N$. This approximation parameter and a finite number of particles,
$P$, are the two approximation parameters that are involved in
approximating a finite average of the QoI computed for all particles
in the system. Then, to approximate the expectation of this average,
we use a Monte Carlo method. In such a method, multiple
\emph{independent} and \emph{identical} stochastic particle systems,
approximated with the same number of time steps, $N$, are simulated
and the average QoI is computed from each and an overall average is
then taken. Using this method, a reduction of the variance of the
estimator is achieved by increasing the number of simulations of the
stochastic particle system or increasing the number of particles in
the system.  Section~\ref{ss:mc} presents the Monte Carlo method more
precisely in the setting of stochastic particle systems.  Particle
methods that are not based on Monte Carlo were also discussed in
\cite{bossy1996convergence, bossy1997stochastic}.  In these methods, a
\emph{single} simulation of the stochastic particle system is carried
out and only the number of particles is increased to reduce the
variance.

As an improvement of Monte Carlo methods, the Multilevel Monte Carlo
(MLMC) method was first introduced in \cite{heinrich:MLMC} for
parametric integration and in \cite{giles:MLMC} for SDEs; see
\cite{giles:acta} and references therein for an overview. MLMC
improves the efficiency of the Monte Carlo method when only an
approximation, controlled with a single discretization parameter, of
the solution to the underlying system can be computed.  The basic idea
is to reduce the number of required samples on the finest, most
accurate but most expensive discretization, by reducing the variability
of this approximation with a \emph{correlated} coarser and cheaper
discretization as a control variate. More details are given in
Section~\ref{ss:mlmc} for the case of stochastic particle systems. The
application of MLMC to particle systems has been investigated in many
works \cite{bujok2013multilevel,hajiali:msthesis,
  rosin2014multilevel}. The same concepts have also been applied to
nested expectations \cite{giles:acta}.
More recently, a particle method applying the MLMC methodology to
stochastic particle systems was also introduced in
\cite{ricketson:mlmcparticle} achieving, for a linear system with a
diffusion coefficient that is independent of the state variable, a
work complexity of $\Order{\tol^{-2} \p{\log\p{\tol^{-1}}}^5}$.

Recently, the Multi-index Monte Carlo (MIMC) method
\cite{abdullatif.etal:MultiIndexMC} was introduced to tackle high
dimensional problems with more than one discretization parameter. MIMC
is based on the same concepts as MLMC and improves the efficiency of
MLMC even further but requires mixed regularity with respect to the
discretization parameters. More details are given in
Section~\ref{ss:mimc} for the case of stochastic particle systems.  In
that section, we demonstrate the improved work complexity of MIMC
compared with the work complexity of MC and MLMC, when applied to a
stochastic particle system. More specifically, we show that, when
using a naive simulation method for the particle system with quadratic
complexity, the optimal work complexity of MIMC is
$\Order{\tol^{-2}\log\p*{\tol^{-1}}^2}$ when using the Milstein
time-stepping scheme and $\Order{\tol^{-2}\log\p*{\tol^{-1}}^4}$ when
using the Euler-Maruyama time-stepping scheme.  Finally, in
Section~\ref{s:res}, we provide numerical verification for the
assumptions that are made throughout the current work and the derived
rates of the work complexity.

In what follows, the notation $a \lesssim b$ means that there
exists a constant $c$ that is independent of $a$ and $b$ such that
$a < cb$.

\section{Problem Setting}\label{sec:prob-set}
Consider a system of $P$ exchangeable stochastic differential equations (SDEs) where for
$p = 1 \ldots P$, we have the following equation for $X_{p|P}(t) \in \rset^d$
\begin{equation}
\label{eq:particle-sys}
\en*\{.{\aligned \di X_{p|P}(t) &= \driftA\op*{t, X_{p|P}(t),
    \empmeasure_{\vec X_{P}(t)}} \di t +
  \volB\op*{t, X_{p|P}(t), \empmeasure_{\vec X_{P}(t)}}
  \di W_{p}(t) \\
  X_{p|P}(0) &= x^0_{p} \endaligned}
\end{equation}
where $\vec X_{P}(t) = \en\{\}{X_{q|P}(t)}_{q=1}^P$ and for some
(possibly stochastic) functions,
$\driftA\::\: [0, \infty) \times \rset^d \times \probspace(\rset^d)
\to \rset^d$ and
$\volB\::\: [0, \infty) \times \rset^d \times \probspace(\rset^d) \to
\rset^d \times\rset^d$ and $\probspace(\rset^d)$ is the space of
probability measures over $\rset^d$. Moreover,
\[
\empmeasure_{\vec X_{P}(t)} \eqdef \frac{1}{P}\sum_{q=1}^P
\delta_{X_{q|P}(t)} \in \probspace(\rset^d),
\]
where $\delta$ is the Dirac measure, is called the empirical measure.  In this setting,
$\en*\{\}{W_{p}}_{p\geq 1}$ are mutually independent $d$-dimensional Wiener
processes. If, moreover, $\{x^0_p\}_{p \geq 1}$ are i.i.d., then under certain conditions on the
smoothness and form of $\driftA$ and $\volB $ \cite{sznitman1991topics}, as
$P \to \infty$ for any $p \in \nset$, the $X_{p|\infty}$ stochastic process satisfies
\begin{equation}
\label{eq:particle-sys-meanfield}
\en*\{.{\aligned \di X_{p|\infty}(t) &= \driftA(t, X_{p|\infty}(t), \mu_\infty^t) \di t +
  \volB(t, X_{p|\infty}(t), \mu_\infty^t) \di W_{p}(t) \\
  X_{p|\infty}(0) &= x^0_{p}, \endaligned}
\end{equation}
where $\mu_\infty^t \in \probspace(\rset^d)$ is the corresponding
mean-field measure.  Under some smoothness and boundedness conditions
on $\driftA$ and $\volB$, the measure $\mu_\infty^t$ induces a
probability density function (pdf), $\rho_\infty(t, \cdot)$, that is
the Radon-Nikodym derivative with respect to the Lebesgue measure.
Moreover, $\rho_\infty$ satisfies the McKean-Vlasov equation
\[
  \frac{\partial \rho_\infty(t, x)}{\partial t} + \sum_{p=1}^P
  \frac{\partial}{\partial x_p}\p{\driftA(t, x_p, \mu_\infty^t)
    \rho_\infty(t, x)} - \frac{1}{2}\sum_{p=1}^P
  \frac{\partial^2}{\partial x_p^2} \p*{ \volB(t, x_p,
    \mu_\infty^t)^2 \rho_\infty(t, x) }= 0
\]
on $t \in [0, \infty)$ and $x = \p{x_p}_{p=1}^P \in \rset^d$ with
$\rho_\infty(0, \cdot)$ being the pdf of $x_p^0$ which is given and is
independent of $p$.  Due to \eqref{eq:particle-sys-meanfield} and
$x_p^0$ being i.i.d, $\{X_{p|\infty}\}_p$ are also i.i.d.; hence,
unless we want to emphasize the particular path, we drop the
$p$-dependence in $X_{p|\infty}$ and refer to the random process
$X_\infty$ instead. In any case, we are interested in computing
$\E{\psi\op*{X_{\infty}(T)}}$ for some given function, $\psi$, and
some final time, $T < \infty$.
\begin{kuramoto}[Fully connected Kuramoto model for synchronized
  oscillators] Throughout this work, we focus on a
  simple, one-dimensional example of \eqref{eq:particle-sys}.
For $p = 1, 2,\ldots, P$, we seek $X_{p|P}(t) \in \rset$ that satisfies
\begin{equation}\label{eq:kuramoto-system}
    \aligned
    \textnormal{d}X_{p|P}(t) &= \op*{\vartheta_p +
      \frac{1}{P}\sum_{q=1}^P \sin\op*{X_{p|P}(t)-X_{q|P}(t)}}
    \textnormal{d}t + \sigma \textnormal{d}W_{p}(t)\\X_{p|P}(0) &= x_{p}^0,
\endaligned
\end{equation}
where $\sigma \in \rset$ is a constant and $\{\vartheta_p\}_p$ are
i.i.d. and independent from the set of i.i.d. random variables
$\{x_p^0\}_p$ and the Wiener processes $\en\{\}{W_p}_{p}$. The
limiting SDE as $P \to \infty$ is
\begin{equation*}
  \aligned
  \textnormal{d}X_{p|\infty}(t) &= \op*{\vartheta_p +
    \int_{-\infty}^\infty \sin\op*{X_{p|\infty}(t)-y}
    \di\mu_{\infty}^t(y) }
  \textnormal{d}t + \sigma \textnormal{d}W_{p}(t)\\X_{p|\infty}(0) &= x_{p}^0.
\endaligned
\end{equation*}
Note that in terms of the generic system \eqref{eq:particle-sys} we
have
\[
  \driftA(t, x, \mu) = \vartheta + \int_{-\infty}^{\infty}\sin(x-y)
  \di \mu(y)
\]
with $\vartheta$ a random variable and $\volB = \sigma$ is a constant.
We are interested in
\[ \textnormal{Total synchronization} =
  \op*{\E{\cos(X_{\infty}(T))}}^2 +
  \op*{\E{\sin(X_{\infty}(T))}}^2, \] a real number between zero and
one that measures the level of synchronization in the system with an
infinite number of oscillators~\cite{acebron:kuramoto}; with zero
corresponding to total disorder. In this case, we need two estimators:
one where we take $\psi(\cdot) = \sin(\cdot)$ and the other where we
take $\psi(\cdot) = \cos(\cdot)$.
\end{kuramoto}

While it is computationally efficient to approximate
$\E{\psi(X_\infty(T))}$ by solving the McKean-Vlasov PDE, that
$\rho_\infty$ satisfies, when the state dimensionality, $d$, is small
(cf., e.g., \cite{hajiali:msthesis}), the cost of a standard full
tensor approximation increases exponentially as the dimensionality of
the state space increases.
On the other hand, using sparse approximation
techniques to solve the PDE requires increasing regularity assumptions
as the dimensionality of the
state space increases.
Instead, in this work, we focus on
approximating the value of $\E{\psi\op{X_\infty}}$ by simulating the
SDE system in \eqref{eq:particle-sys}. Let us now define
\begin{equation} \label{eq:phi-P}
\phi_P \eqdef  \frac{1}{P} \sum_{p=1}^P
\psi\op*{X_{p|P}(T)}.
\end{equation}
Here, due to exchangeability, $\{X_{p|P}(T)\}_{p=1}^P$ are identically
distributed but they are not independent since they are taken from the
same realization of the particle system. Nevertheless, we have
$\E{\phi_P} = \E{\psi(X_{p|P}(T))}$ for any $p$ and $P$.
In this case, with respect to the number of particles, $P$, the cost
of a naive calculation of $\phi_P$ is $\Order{P^2}$ due to the cost of
evaluating the empirical measure in \eqref{eq:particle-sys} for every
particle in the system. It is possible to take $\{X_{p|P}\}_{p=1}^P$
in \eqref{eq:phi-P} as i.i.d., i.e., for each $p=1\ldots P$, $X_{p|P}$
is taken from a different independent realization of the system
\eqref{eq:particle-sys}. In this case, the usual law of large numbers
applies, but the cost of a naive calculation of $\phi_P$ is
$\Order{P^3}$. For this reason, we focus in this work on the former
method of taking identically distributed but not independent
$\{X_{p|P}\}_{p=1}^P$.

Following the setup in
\cite{abdullatif:continuationMLMC,abdullatif:meshMLMC}, our objective
is to build a random estimator, $\mathcal{A}$, approximating
$\phi_\infty \eqdef \E{\psi(X_\infty(T))}$ with minimal work, i.e., we wish to satisfy the
constraint
  \begin{equation}\label{eq:full-stat-const}
    \prob{ |\mathcal{A} - \phi_\infty| \geq \tol }\
    \le \epsilon
  \end{equation}
  for a given error tolerance, $\tol$, and a given confidence level determined
  by $0<\epsilon \ll 1$.  We instead impose the following, more
  restrictive, two constraints:
  \begin{align}
    \label{eq:bias-const}\text{\bf Bias constraint:} & &|\E{\mathcal{A}-\phi_\infty}| \leq (1-\theta){\tol}, \\
    \label{eq:stat-const}\text{\bf Statistical constraint:}& &\prob
                                                 {|\mathcal{A} - \E{\mathcal{A}}| \geq \theta \tol} \le \epsilon,
  \end{align}
  for a given tolerance splitting parameter, $\theta \in (0,1)$,
  possibly a function of $\tol$. To show that these bounds are
  sufficient note that
    \[\aligned
    \prob{\abs{\mathcal A - \phi_\infty} \geq \tol} &\leq
    \prob{\abs{\mathcal A - \E{\mathcal A }} + \abs{\E{\mathcal A} -
        \phi_\infty} \geq \tol}
    \endaligned\]
  imposing \eqref{eq:bias-const}, yields
  \[ \prob{\abs{\mathcal A - \phi_\infty} \geq \tol} \leq
    \prob{\abs{\mathcal A - \E{\mathcal A }} \geq \theta \tol} \] then
  imposing \eqref{eq:stat-const} gives \eqref{eq:full-stat-const}.
  Next, we can use Markov inequality and impose
  ${\var{\mathcal{A}}} \leq \epsilon (\theta\tol)^2$ to satisfy
  \eqref{eq:stat-const}. However, by assuming (at least asymptotic)
  normality of the estimator, $\mathcal{A}$ we can get a less
  stringent condition on the variance as follows:
  \begin{align}
    \label{eq:var-const} \text{\bf Variance constraint:}& &\var{\mathcal{A}} \leq
                                       \op*{ \frac{\theta \tol}{C_\epsilon} }^2.
  \end{align}
  Here, $0<C_\epsilon$ is such that
  $\Phi(C_\epsilon) = 1 - \frac{\epsilon}{2}$, where $\Phi$ is the
  cumulative distribution function of a standard normal random
  variable, e.g., $C_\epsilon \approx 1.96$ for $\epsilon = 0.05$.
  The asymptotic normality of the estimator is usually shown using
  some form of the Central Limit Theorem (CLT) or the Lindeberg-Feller
  theorem (see, e.g., \cite{abdullatif:continuationMLMC,
    abdullatif.etal:MultiIndexMC} for CLT results for the MLMC and
  MIMC estimators and Figure~\ref{fig:error-vs-tol}-\emph{right}).

  As previously mentioned, we wish to approximate the values of
  $X_\infty$ by using \eqref{eq:particle-sys} with a finite number of
  particles, $P$.  For a given number of particles, $P$, a solution to
  \eqref{eq:particle-sys} is not readily available. Instead, we have
  to discretize the system of SDEs using, for example, the
  Euler-Maruyama time-stepping scheme with $N$ time steps. For
  $n = 0,1, 2, \ldots N-1$,
  \[ \aligned X^{n+1|N}_{p|P} - X^{n|N}_{p|P} &= \driftA\op*{X^{n|N}_{p|P}, \empmeasure_{\vec X^{n|N}_{P}}}
    \frac{T}{N} + \volB\op*{X^{n|N}_{p|P}, \empmeasure_{\vec
        X^{n|N}_{P}}}
    \Delta W_{p}^{n|N} \\
    X^{0|N}_{p|P} &= x^0_{p}, \endaligned \] where
  $\vec X_{P}^{n|N} = \{X_{p|P}^{n|N}\}_{p=1}^P$ and
  $\Delta W_{p}^{n|N} \sim \mathcal N\op*{0, \frac{T}{N}}$ are i.i.d.
  For the remainder of this work, we use the notation
\[ \phi_P^N \eqdef \frac{1}{P} \sum_{p=1}^P
  \psi\op*{X_{p|P}^{N|N}}. \]
At this point, we make the following assumptions:
\begin{align}
\label{eq:P-bias-assumption} \tag{\textbf{P1}}  \abs*{\E{\phi_P^N -
  \psi(X_{\infty})}} \leq \abs*{\E{\psi(X_{\cdot | P}^{N|N}) - \psi(X_{\cdot | P})}}
  + \abs*{\E{\psi(X_{\cdot | P}) - \psi(X_{\infty})}} &\lesssim N^{-1} + P^{-1}, \\
\label{eq:P-var-assumption} \tag{\textbf{P2}}  \var{\phi_P^N} &\lesssim P^{-1}.
\end{align}
These assumption will be verified numerically in Section~\ref{s:res}.
In general, they translate to smoothness and boundedness assumptions
on $\driftA, \volB$ and $\psi$.  %
Indeed, in \eqref{eq:P-bias-assumption}, the weak convergence of the
Euler-Maruyama method with respect to the number of time steps is a
standard result shown, for example, in \cite{kloden:numsde} by
assuming 4-time differentiability of $\driftA, \volB$ and
$\psi$. Showing that the constant multiplying $N^{-1}$ is bounded for
all $P$ is straightforward by extending the standard proof of weak
convergence the Euler-Maruyama method in \cite[Chapter
14]{kloden:numsde} and assuming boundedness of the derivatives
$\driftA, \volB$ and $\psi$.
On the other hand, the weak convergence with respect to the number of
particles, i.e., $\E{\psi(X_{p|P})} \to \E{\psi(X_{\infty})}$ is a
consequence of the propagation of chaos which is shown, without a
convergence rate, in \cite{sznitman1991topics} for $\psi$ Lipschitz,
$\volB$ constant and $\driftA$ of the the form
\begin{equation}
  \driftA(t, x, \mu) = \int \kappa(t, x, y) \mu(\di y)\label{eq:drift-example}
\end{equation}
where $\kappa(t, \cdot, \cdot)$ is Lipschitz. On the other hand, for
one-dimensional systems and using the results from \cite[Theorem
3.2]{kolokoltsov2015mean} we can show the weak convergence rate with
respect to the number of particles and the convergence rate for the
variance of $\phi_P$ as the following lemma shows. Below, $C(\rset)$
is the space of continuous bounded functions and $C^k(\rset)$ is the
space of continuous bounded functions whose $i$'th derivative is in
$C(\rset)$ for $i=1, \ldots, k$.

\begin{lemma}[Weak and variance convergence rates w.r.t. number of
  particles]
  Consider \eqref{eq:particle-sys} and
  \eqref{eq:particle-sys-meanfield} with $d=1$, strictly positive
  $\volB(t, \cdot, \mu) = \volB(\cdot) \in C^3(\rset)$ and $\driftA$
  as in \eqref{eq:drift-example} with
  $\kappa(t, x, \cdot) \in C^{2}(\rset)$,
  $\frac{\partial \kappa(t, x, \cdot)}{\partial x} \in C(\rset)$ and
  $\kappa(t, \cdot, y) \in C^{2}(\rset)$ where the norms are assumed
  to be uniform with respect the arguments, $x$ and $y$,
  respectively. If, moreover, $\psi \in C^{2}(\rset)$, then
  \begin{align}
    \abs*{\E{\psi\p{X_{\cdot|P}} - \E{\psi\p{X_{\infty}}}}} &\lesssim
                                                              P^{-1},\label{eq:P-weak-conv}   \\
    {\var{\frac{1}{P}\sum_{p=1}^P \psi(X_{p|P})}} &\lesssim P^{-1}.
    \label{eq:P-var-conv}
  \end{align}
\end{lemma}
\begin{proof}
  The system in this lemma is a special case of the system in
  \cite[Theorem 3.2]{kolokoltsov2015mean}. From there and given the
  assumptions of the current lemma, \eqref{eq:P-weak-conv} immediately
  follows. Moreover, from the same reference, we can futher conclude that
  \[
\abs*{\E{\psi(X_{p|P})\psi(X_{q|P})} - \E{\psi(X_{\infty})}^2}
\lesssim {P^{-1}}
\]
for $1 \leq p\neq q \leq P$.
Using this we can show \eqref{eq:P-var-conv} since
\[\aligned
  {\var{\frac{1}{P}\sum_{p=1}^P \psi(X_{p|P})}} &=
  \frac{1}{P}\var{\psi\p*{X_{\cdot|P}}} + \frac{1}{P^2}
  \sum_{p=1}^P\sum_{q=1, p\neq q}^P\cov{\psi(X_{p|P}), \psi(X_{q|P})}
  \endaligned
\]
and
\[\aligned
  \cov{X_{p|P}, X_{q|P}} &= \E{\psi(X_{p|P}) \psi(X_{q|P})} -
  \E{\psi(X_{\cdot|P})}^2 \\
  &= \E{\psi(X_{p|P}) \psi(X_{q|P})} - \E{\psi(X_{\infty})}^2 \\
  &\hskip 1cm -\p{\E{\psi(X_{\cdot|P})} - \E{\psi(X_{\infty})}}^2 -
  2\E{\psi(X_{\infty})}\p{\E{\psi(X_{\cdot|P})} -
    \E{\psi(X_{\infty})}}
  \\
  &\lesssim {P^{-1}}.  \endaligned\]
\end{proof}

From here, the rate of convergence for the variance of $\phi_P^N$ can
be shown by noting that
\[
  \var{\phi_P^N} \leq \abs*{\var{\phi_P^N}-\var{\phi_P}} +
  \var{\phi_P}
\]
and noting that $\var{\phi_P} \lesssim {P^{-1}}$, then showing that
the first term is
$\abs*{\var{\phi_P^N}-\var{\phi_P}} \lesssim {N^{-1} P^{-1}}$ because
of the weak convergence with respect to the number of time steps.

Finally, as mentioned above, with a naive method, the total cost to
compute a single sample of $\phi_{P}^N$ is $\Order{N P^2}$. The
quadratic power of $P$ can be reduced by using, for example, a
multipole algorithm \cite{carrier1988fast, greengard1987fast}.  In
general, we consider the work required to compute one sample of
$\phi_P^N$ as $\Order{N P^{\ggp}}$ for a positive constant,
$\ggp \geq 1$.

\section{Monte Carlo methods} \label{s:mc} In this section, we study
different Monte Carlo methods that can be used to estimate the
previous quantity, $\phi_\infty$. In the following, we use the
notation
$\vec \omega_{p:P}^{(m)} \eqdef \op*{\omega^{(m)}_{q}}_{q=p}^P$ where,
for each $q$, $\omega^{(m)}_{q}$ denotes the $m$'th sample of the set
of underlying random variables that are used in calculating
$X_{q|P}^{N|N}$, i.e., the Wiener path, $W_{q}$, the initial
condition, $x_q^0$, and any random variables that are used in
$\driftA$ or $\volB$. Moreover, we sometimes write
$\phi_P^N(\vec \omega_{1:P}^{(m)})$ to emphasize the dependence of the
$m'$th sample of $\phi_P^N$ on the underlying random variables.

\subsection{Monte Carlo (MC)} \label{ss:mc}
The first estimator that we look at is a Monte Carlo estimator. For a
given number of samples, $M$, number of particles, $P$, and number of
time steps, $N$, we can write the MC estimator as follows:
\[ \mathcal A_{\text{MC}}(M, P, N) = \frac{1}{M} \sum_{m=1}^M
  \phi_P^N(\vec \omega_{1:P}^{(m)}). \]
Here,
\[
\aligned
&\E{\mathcal A_{\text{MC}}(M, P, N)} = \E{\phi_P^N} =
\frac{1}{P}\sum_{p=1}^P \E{\psi(X_{p|P}^{N|N})} = \E{\psi(X_{\ \cdot|P}^{N|N})},\\
\text{and}\quad& \var{\mathcal A_{\text{MC}}(M, P, N)} = \frac{\var{\phi_P^N}}{M},\\
\text{while the total work is}\quad& \work{\mathcal A_{\text{MC}}(M,
  P,N)} = M  NP^{\ggp}.
\endaligned\]
Hence, due to \eqref{eq:P-bias-assumption}, we must have
$P = \Order{\tol^{-1}}$ and $N = \Order{\tol^{{-1}}}$ to satisfy
\eqref{eq:bias-const}, and, due to \eqref{eq:P-var-assumption}, we
must have $M = \Order{\tol^{-1}}$ to satisfy \eqref{eq:var-const}.
Based on these choices, the total work to compute $\mathcal A_{\text{MC}}$ is
\[ \work{\mathcal A_{\text{MC}}} = \Order{\tol^{-2 -\ggp}}. \]
\begin{kuramoto}
  Using a naive calculation method of $\phi_{P}^N$ (i.e., $\ggp=2$)
  gives a work complexity of $\Order{\tol^{-4}}$.  See also
  Table~\ref{tbl:rates} for the work complexities for different common
  values of $\ggp$.
\end{kuramoto}

\subsection{Multilevel Monte Carlo (MLMC)} \label{ss:mlmc} For a given
$L \in \nset$, define two hierarchies, $\{N_\ell\}_{\ell=0}^L$ and
$\{P_\ell\}_{\ell=0}^L$, satisfying $P_{\ell-1} \leq P_{\ell}$ and
$N_{\ell-1} \leq N_{\ell}$ for all $\ell$.  Then, we can write the
MLMC estimator as follows:
\begin{equation}
\label{eq:MLMC-estimator}
  \mathcal A_{\text{MLMC}}(L) = \sum_{\ell=0}^L
  \frac{1}{M_\ell} \sum_{m=1}^{M_\ell}
  {\op*{\phi_{P_\ell}^{N_\ell} - \varphi_{P_{\ell-1}}^{N_{\ell-1}}}\op*{\vec \omega_{1:P_\ell}^{(\ell,m)}}},
\end{equation}
where we later choose the function
$\varphi_{P_{\ell-1}}^{N_{\ell-1}}(\cdot)$ such that
$\varphi_{P_{-1}}^{N_{-1}}(\cdot) = 0$ and
$ \E{\varphi_{P_{\ell-1}}^{N_{\ell-1}}}=
\E{\phi_{P_{\ell-1}}^{N_{\ell-1}}},$
so that $\E{\mathcal A_{\textnormal{MLMC}}} = \E{\phi_{P_L}^{N_L}}$
due to the telescopic sum.  For MLMC to have better work complexity
than that of Monte Carlo,
$\phi_{P_\ell}^{N_\ell}(\vec \omega_{1:P_\ell}^{(\ell,m)})$ and
$\varphi_{P_{\ell-1}}^{N_{\ell-1}}(\vec \omega_{1:P_\ell}^{(\ell,m)})$
must be correlated for every $\ell$ and $m$, so that their difference
has a smaller variance than either
$\phi_{P_\ell}^{N_\ell}(\vec \omega_{1:P_\ell}^{(\ell,m)})$ or
$\varphi_{P_{\ell-1}}^{N_{\ell-1}}(\vec \omega_{1:P_\ell}^{(\ell,m)})$ for all
$\ell > 0$.

Given two discretization levels, $N_\ell$ and $N_{\ell-1}$, with the
same number of particles, $P$, we can generate a sample of
$\varphi_{P}^{N_{\ell-1}}(\vec \omega_{1:P}^{(\ell,m)})$ that is correlated to
$\phi_{P}^{N_\ell}(\vec \omega_{1:P}^{(\ell,m)})$ by taking
\[ \varphi_{P}^{N_{\ell-1}}(\vec \omega_{1:P}^{(\ell,m)}) =
  \phi_{P}^{N_{\ell-1}}(\vec \omega_{1:P}^{(\ell,m)}).\] That is, we
use the same samples of the initial values, $\{x_{p}^0\}_{p\geq 1}$,
the same Wiener paths, $\{W_p\}_{p=1}^P$, and, in case they are random
as in \eqref{eq:kuramoto-system}, the same samples of the advection
and diffusion coefficients, $\driftA$ and $\volB$, respectively.
We can improve the correlation by using an antithetic sampler as
detailed in \cite{giles:antithetic} or by using a higher-order scheme
like the Milstein scheme \cite{giles:milstein}. In the Kuramoto
example, the Euler-Maruyama and the Milstein schemes are equivalent
since the diffusion coefficient is constant.

On the other hand, given two different sizes of the particle system,
$P_\ell$ and $P_{\ell-1}$, with the same discretization level, $N$, we
can generate a sample of
$\varphi_{P_{\ell-1}}^{N}(\vec \omega_{1:P_\ell}^{(\ell,m)})$
that is correlated to
$\phi_{P_\ell}^{N}(\vec \omega_{1:P_\ell}^{(\ell,m)})$
by taking
\begin{equation}
\label{eq:MLMC-non-antithetic}
 \varphi_{P_{\ell-1}}^{N}(\vec \omega_{1:P_{\ell}}^{(\ell,m)}) =
\overline{\varphi}_{P_{\ell-1}}^{N}(\vec \omega_{1:P_{\ell}}^{(\ell,m)}) \eqdef
  \phi_{P_{\ell-1}}^{N}\op*{\vec \omega_{1:P_{\ell-1}}^{(\ell,m)}}.
\end{equation}
In other words, we use the same $P_{\ell-1}$ sets of random variables
out of the total $P_\ell$ sets of random variables to run an
independent simulation of the stochastic system with $P_{\ell-1}$
particles.

We also consider another estimator that is more correlated with
$\phi_{P_\ell}^{N}(\vec \omega_{1:P_\ell}^{(\ell,m)})$. The
``antithetic'' estimator was first independently introduced in
\cite[Chapter 5]{hajiali:msthesis} and \cite{bujok2013multilevel} and
subsequently used in other works on particle systems
\cite{rosin2014multilevel} and nested simulations \cite{giles:acta}.
In this work, we call this estimator a ``partitioning'' estimator to
clearly distinguish it from the antithetic estimator in
\cite{giles:antithetic}.  We assume that $P_\ell = \betap P_{\ell-1}$
for all $\ell$ and some positive integer $\betap$ and take
\begin{equation}
\label{eq:MLMC-antithetic}
 \varphi_{P_{\ell-1}}^{N}(\vec \omega_{1:P_\ell}^{(\ell,m)}) =
 \widehat\varphi_{P_{\ell-1}}^{N}(\vec \omega_{1:P_\ell}^{(\ell,m)}) \eqdef
\frac{1}{\betap}
  \sum_{i=1}^{\betap} \phi_{P_{\ell-1}}^{N}
  \op*{
\vec \omega^{(\ell,m)}_{\op*{(i-1)P_{\ell-1}+1}\ :\ iP_{\ell-1}}
}.
\end{equation}
That is, we split the underlying $P_\ell$ sets of random variables
into $\betap$ identically distributed and independent groups,
each of size $P_{\ell-1}$, and independently simulate $\betap$
particle systems, each of size $P_{\ell-1}$. Finally, for each
particle system, we compute the quantity of interest and take the
average of the $\betap$ quantities.

\providecommand{\other}[1]{{\widetilde{#1}}}
In the following subsections, we look at different settings in
which either $P_\ell$ or $N_\ell$ depends on $\ell$ while the other
parameter is constant for all $\ell$. We begin by recalling the
optimal convergence rates of MLMC when applied to a generic random
variable, $Y$, with a trivial generalization to the case when there
are two discretization parameters: one that is a function of the
level, $\ell$, and the other, $\other{L}$, that is fixed for all levels.
\begin{theorem}[Optimal MLMC complexity] \label{thm:mlmc-complexity}
  Let $Y_{\other{L},  \ell}$ be an approximation of the
  random variable, $Y$, for every $(\other{L}, \ell) \in \nset^2$.
  Denote by $Y^{(\ell, m)}$ a sample of $Y$ and denote its corresponding
  approximation by $Y_{\other{L}, \ell}^{(\ell, m)}$, where we
  assume that the samples $\{Y^{(\ell, m)}\}_{\ell, m}$ are mutually
  independent.  Consider the MLMC estimator
  \[ \mathcal A_{\textnormal{MLMC}}(\other{L}, L) = \sum_{\ell=0}^{L}
    \frac{1}{M_\ell}\sum_{m=1}^{M_\ell} (Y_{\other{L}, \ell}^{(\ell,
      m)} - Y_{\other{L}, \ell-1}^{(\ell, m)}) \] with
  $Y_{\other{L}, -1}^{\ell, m} = 0$ and for
  $\beta, w, \gamma, s, \other{\beta}, \other{w}, \other{\gamma},
  \other{c} > 0$ where $s \leq 2 w$, assume the following:
\begin{enumerate}
\item
  $\abs*{\E{Y - Y_{\other{L}, \ell}}} \lesssim \other{\beta}^{-\other{w}
  \other{L}} + \beta^{-w \ell} $
\item
  $\var{Y_{\other{L}, \ell} - Y_{\other{L}, \ell-1}} \lesssim
  \other{\beta}^{-\other{c} \other{L}} \beta^{-s \ell}$
\item $\work{Y_{\other{L}, \ell} - Y_{\other{L}, \ell-1}} \lesssim
  \other\beta^{\other{\gamma} \other{L}} \beta^{\gamma \ell}$.
\end{enumerate}
Then, for any
$\tol < e^{-1}$, there exists $\other{L}, L$ and a sequence of
$\en*\{\}{M_\ell}_{\ell=0}^L$ such that
  \begin{equation}
    \label{eq:mlmc-thm-prob-constraint}
  \prob{\abs{\mathcal A_{\textnormal{MLMC}}(\other{L}, L) - Y} \geq
    \tol} \leq \epsilon
\end{equation}
and
\begin{equation}\aligned
  \label{eq:mlmc-work}
  \work{\mathcal A_{\textnormal{MLMC}}(\other L, L)} &\eqdef
  \sum_{\ell=0}^L M_\ell \work{Y_{\other{L}, \ell} - Y_{\other{L},
      \ell-1}} \\
  &\lesssim \begin{cases}
    \tol^{- 2 - {\frac{\other{ \gamma}-\other{c}}{\other{ w}}} } & \text{if } s > \gamma
    \\
    \tol^{- 2 - {\frac{\other{\gamma}-\other{c}}{\other{w}}} } \log\op*{\tol^{-1}}^2 & \text{if } s = \gamma
    \\
    \tol^{- 2 - {\frac{\other{\gamma}-\other{c}}{\other{w}}} - \frac{\gamma-s}{w}} & \text{if } s < \gamma.
\end{cases}
\endaligned\end{equation}
\end{theorem}
\begin{proof}
  The proof can be straightforwardly derived from the proof of
  \cite[Theorem 1]{scheichl.giles:MLMC}, we sketch here the main
  steps.  First, we split the constraint
  \eqref{eq:mlmc-thm-prob-constraint} to a bias and variance
  constraints similar to \eqref{eq:bias-const} to
  \eqref{eq:var-const}, respectively. Then, since
  $\E{\mathcal A_{\textnormal{MLMC}}(\other{L}, L)} = \E{Y_{\other L,
      L}}$, given the first assumption of the theorem and
  imposing the bias constraint yield
  $\other L = \Order{\frac{1}{\other{w}\log(\other{\beta})}\log\p{\tol^{-1}}}$ and
  $L= \Order{\frac{1}{w\log(\beta)}\log\p{\tol^{-1}}}$. The assumptions on the
  variance and work then give:
  \[\aligned
    \var{Y_{\other{L}, \ell} - Y_{\other{L}, \ell-1}} &\lesssim
    \tol^{\frac{\other{c}}{\other{w}}} \beta^{-s \ell},\\
    \work{Y_{\other{L}, \ell} - Y_{\other{L}, \ell-1}} &\lesssim
    \tol^{-\frac{\other{\gamma}}{\other{w}}} \beta^{\gamma \ell}.
    \endaligned
  \]
  Then
  \[
    \var{\mathcal A_{\textnormal{MLMC}}(\other{L}, L)} =
    \sum_{\ell=0}^L M_\ell^{-1} {\var{Y_{\other{L}, \ell} -
        Y_{\other{L}, \ell-1}}} \lesssim
    \tol^{\frac{\other{c}}{\other{w}}} \sum_{\ell=0}^L
    M_\ell^{-1} \beta^{-s \ell},
  \]
  due to mutual independence of $\{Y^{(\ell, m)}\}_{\ell, m}$. Moreover,
  \[
    \work{\mathcal A_{\textnormal{MLMC}}(\other{L}, L)} =
    \sum_{\ell=0}^L M_\ell \work{Y_{\other{L}, \ell} - Y_{\other{L},
        \ell-1}} \lesssim \tol^{-\frac{\other{\gamma}}{\other{w}}}
    \sum_{\ell=0}^L M_\ell^{-1} \beta^{\gamma \ell}
  \]
  Finally, given $L$, solving for $\{M_\ell\}_{\ell=0}^L$ to minimize
  the work while satisfying the variance constraint gives the desired
  result.
\end{proof}

\subsubsection{MLMC hierarchy based on the number of time steps}\label{sec:mlmc-N}
In this setting, we take $N_\ell = \p{\betat}^{\ell}$ for some $\betat>0$
and $P_\ell = P_L$ for all $\ell$, i.e., the number of particles is a
constant, $P_L$, on all levels.  We make an extra assumption in this
case, namely:
\begin{equation}
  \label{eq:MLMC-N-var-assumption}
  \tag{\textbf{MLMC1}}\var{\phi_{P_L}^{N_\ell} -
    \varphi_{P_L}^{N_{\ell-1}}} \lesssim P_L^{-1} N_{\ell}^{-\sst} =
  P_L^{-1} \op{\betat}^{-\sst \ell},
\end{equation}
for some constant $\sst > 0$. The factor $\op{\betat}^{-\sst \ell}$ is
the usual assumption on the variance convergence of the level
difference in MLMC theory \cite{giles:MLMC} and is a standard result
for the Euler-Maruyama scheme with $\sst = 1$ and for the Milstein
scheme with $\sst =2$, \cite{kloden:numsde}.  On the other hand, the
factor $P_L^{-1}$ can be motivated from \eqref{eq:P-var-assumption},
which states that the variance of each term in the difference
converges at this rate.

Due to Theorem~\ref{thm:mlmc-complexity}, we can conclude that the
work complexity of MLMC is
\begin{equation}
  \label{eq:mlmc-n-compelxity}
  \work{\mathcal A_{\textnormal{MLMC}}} \lesssim \begin{cases}
    \tol^{- 1  - \ggp} & \text{if } \sst > 1
    \\
    \tol^{- 1 - \ggp} \log\op*{\tol^{-1}}^2 & \text{if } \sst = 1
    \\
    \tol^{- 2 - \ggp + \sst}
    & \text{if } \sst < 1.
\end{cases}
\end{equation}

\begin{kuramoto}
  In this example, using the Milstein time-stepping scheme, we have
  $\sst = 2$ (cf. Figure~\ref{fig:mlmc-rates}), and a naive
  calculation method of $\phi_{P}^N$ ($\ggp=2$) gives a work
  complexity of $\Order{\tol^{-3}}$.  See also Table~\ref{tbl:rates}
  for the work complexities for different common values of $\sst$ and
  $\ggp$.
\end{kuramoto}

\subsubsection{MLMC hierarchy based on the number of
  particles}\label{sec:mlmc-P}
In this setting, we take  $P_\ell = \p{\betap}^{\ell}$ for some $\betap>0$
and $N_\ell = N_L$ for all $\ell$, i.e., we take the number of
time steps to be a constant, $N_L$, on all levels. We make an extra assumption in this case:
\begin{equation}
  \label{eq:MLMC-P-var-assumption}
  \tag{\textbf{MLMC2}}\var{\phi_{P_{\ell}}^{N_L} -
    \varphi_{P_{\ell-1}}^{N_{L}}} \lesssim P_{\ell}^{-\ssp-1}
  = \betap^{-\ell(\ssp+1)},
\end{equation}
for some constant $\ssp\geq 0$.  The factor $\betap^{-\ssp \ell}$ is
the usual assumption on the variance convergence of the level
difference in MLMC theory \cite{giles:MLMC}.
On the other hand, the factor $P_\ell^{-1}$ can be motivated
from \eqref{eq:P-var-assumption}, since the variance of each term in
the difference is converging at this rate.

Due to Theorem~\ref{thm:mlmc-complexity}, we can conclude that the
work complexity of MLMC in this case is
\begin{equation}
  \work{\mathcal A_{\textnormal{MLMC}}} \lesssim \begin{cases}
    \tol^{- 3} & \text{if } \ssp+1 > \ggp
    \\
    \tol^{- 3} \log\op*{\tol^{-1}}^2 & \text{if } \ssp+1 = \ggp
    \\
    \tol^{- 2 - {\ggp+\ssp}}
    & \text{if } \ssp+1 < \ggp.
\end{cases}
\end{equation}

\begin{kuramoto}
Using a naive calculation method of $\phi_{P}^N$ ($\ggp=2$),
we distinguish between the two samplers:
\begin{itemize}
\item Using the sampler $\overline \varphi$
  in~\eqref{eq:MLMC-non-antithetic}, we verify
  numerically that
  $\ssp = 0$ (cf. Figure~\ref{fig:mlmc-rates}). Hence, the work complexity is
  \( \Order{\tol^{-4}}, \)
  which is the same work complexity as a Monte Carlo estimator. This should
  be expected since using the ``correlated'' samples of
  $\overline \varphi_{P_{\ell-1}}^N$ and $\phi_{P_\ell}^N$ do not
  reduce the variance of the difference, as
  Figure~\ref{fig:mlmc-rates} shows.
\item Using the partitioning estimator, $\widehat \varphi$,
  in~\eqref{eq:MLMC-antithetic}, we verify numerically that $\ssp = 1$
  (cf. Figure~\ref{fig:mlmc-rates}).  Hence, the work complexity is
  $ \Order{\tol^{-3} \log\op*{\tol^{-1}}^2}$.  Here the samples of
  $\widehat \varphi_{P_{\ell-1}}^N$ have higher correlation to
  corresponding samples of $\phi_{P_\ell}^N$, thus reducing the
  variance of the difference. Still, using MLMC with hierarchies based
  on the number of times steps (fixing the number of particles) yields
  better work complexity. See also Table~\ref{tbl:rates} for the work
  complexities for different common values of $\sst$ and $\ggp$.
\end{itemize}
\end{kuramoto}

\subsubsection{MLMC hierarchy based on both the number of particles and
  the number of times steps}\label{sec:mlmc-NP}
In this case, we vary both the number of particles and the number of
time steps across MLMC levels. That is, we take
$P_\ell = \p{\betap}^{\ell}$ and $N_\ell = \p{\betat}^{\ell}$ for all
$\ell$. In this case, a reasonable assumption is

\begin{equation}
  \label{eq:MLMC-NP-var-assumption}
  \tag{\textbf{MLMC3}}\var{\phi_{P_\ell}^{N_\ell} -
    \phi_{P_{\ell-1}}^{N_{\ell-1}}} \lesssim
  \p{\betap}^{-\ell} \op*{\max\p*{\p{\betap}^{-\ssp}, \p{\betat}^{-\sst}}}^\ell.
\end{equation}
The factor $\betap^{-\ell}$ can be motivated
from~\eqref{eq:P-var-assumption} since the variance of each term in
the difference is converges at this rate. On the other hand,
$\op*{\max(\betap^{-\ssp}, \betat^{-\sst})}^\ell$ is the larger factor
of~\eqref{eq:MLMC-N-var-assumption}
and~\eqref{eq:MLMC-P-var-assumption}.

Due to Theorem~\ref{thm:mlmc-complexity} and defining
\[
\aligned
s &= \log(\betap)+\min(\ssp \log(\betap), \sst \log(\betat)) \\
\gamma &= \ggp \log(\betap)+ \log(\betat)\\
w &= \min(\log(\betap), \log(\betat)),
\endaligned
\]
we can conclude that the work complexity of MLMC is
\begin{equation}\label{eq:mlmc-NP-complexity}
  \work{\mathcal A_{\textnormal{MLMC}}} \lesssim \begin{cases}
    \tol^{- 2 } & \text{if }
     s > \gamma
    \\
    \tol^{- 2}  \log\op*{\tol^{-1}}^2 & \text{if } s = \gamma
    \\
    \tol^{- 2- \frac{\gamma-s}{w}}
    & \text{if } s < \gamma.
\end{cases}
\end{equation}
\begin{kuramoto}
  We choose $\betap=\betat$ and use a naive calculation method of
  $\phi_{P}^N$ (yielding $\ggp=2$) and the partitioning sampler
  (yielding $\ssp=1$). Finally, using the Milstein time-stepping
  scheme, we have $\sst=2$. Refer to Figure~\ref{fig:mlmc-rates} for
  numerical verification. Based on these rates, we have, in
  \eqref{eq:mlmc-NP-complexity}, $s=2\log(\betap), w=\log(\betap)$ and
  $\gamma=3\log(\betap)$. The MLMC work complexity in this case is
  \( \Order{\tol^{-3}}. \) See also Table~\ref{tbl:rates} for the work
  complexities for different common values of $\sst$ and $\ggp$.
\end{kuramoto}

\subsection{Multi-index Monte Carlo (MIMC)} \label{ss:mimc} Following
\cite{abdullatif.etal:MultiIndexMC}, for every multi-index
$\valpha = (\alpha_1, \alpha_2) \in \nset^{2}$, let
$P_{\alpha_1} = \p{\betap}^{\alpha_1}$ and
$N_{\alpha_2} = \p{\betat}^{\alpha_2}$ and define the first-order
mixed-difference operator in two dimensions as
\[ \vec \Delta \phi_{P_{\alpha_1}}^{N_{\alpha_2}}\op*{\vec
    \omega_{1:P_{\alpha_1}}} =
\p*{\op*{\phi_{P_{\alpha_1}}^{N_{\alpha_2}} - \varphi_{P_{\alpha_1
        -1}}^{N_{\alpha_2}} } -
  \op*{\phi_{P_{\alpha_1}}^{N_{\alpha_2-1}} -
    \varphi_{P_{\alpha_1-1}}^{N_{\alpha_2-1}} }}\p*{\vec
  \omega_{1:P_{\alpha_1}}}\] with $\phi_{P_{-1}}^{N}=0$ and
$\phi_{P}^{N_{-1}}=0$.  The MIMC estimator is then written for a given
$\mathcal I \subset \nset^{2}$ as
\begin{equation}
  \label{eq:mimc-estimator}
  \mathcal A_{\textnormal{MIMC}} = \sum_{\valpha \in \mathcal I}
  \frac{1}{M_\valpha}
\sum_{m=1}^{M_\valpha} \vec \Delta \phi_{P_{\alpha_1}}^{N_{\alpha_2}}
\p*{\vec \omega_{1:P_{\alpha_1}}^{(\valpha, m)}}
\end{equation}
At this point, similar to the original work on MIMC
\cite{abdullatif.etal:MultiIndexMC}, we make the following assumptions
on the convergence of
$\vec \Delta \phi_{P_{\alpha_1}}^{N_{\alpha_2}} $, namely
\begin{align}
\label{eq:MIMC-bias-assumption}\tag{\textbf{MIMC1}}\E{\vec \Delta \phi_{P_{\alpha_1}}^{N_{\alpha_2}}} \lesssim
  P_{\alpha_1}^{-1} N_{\alpha_2}^{-1}\\
\label{eq:MIP-var-assumption}\tag{\textbf{MIMC2}}\var{\vec \Delta \phi_{P_{\alpha_1}}^{N_{\alpha_2}}} \lesssim
P_{\alpha_1}^{-\ssp-1} N_{\alpha_2}^{-\sst}.%
\end{align}
Assumption \eqref{eq:MIMC-bias-assumption} is motivated from
\eqref{eq:P-bias-assumption} by assuming that the mixed first order
difference, $\vec \Delta \phi_{P_{\alpha_1}}^{N_{\alpha_2}}$, gives a
product of the convergence terms instead of a sum. Similarly,
\eqref{eq:MIP-var-assumption} is motivated from
\eqref{eq:MLMC-N-var-assumption} and
\eqref{eq:MLMC-P-var-assumption}. To the best of our knowledge, there
are currently no proofs of these assumptions for particle systems, but
we verify them numerically for~\eqref{eq:kuramoto-system} in
Figure~\ref{fig:mimc-rates}.

Henceforth, we will assume that $\betat = \betap$ for easier
presentation. Following \cite[Lemma 2.1]{abdullatif.etal:MultiIndexMC}
and recalling the assumption on cost per sample,
$ \work{\vec \Delta \phi_{P_{\alpha_1}}^{N_{\alpha_2}}} \lesssim
P_{\alpha_1}^{\ggp} N_{\alpha_2}^{} $, then, for every value of
$L\in\rset^+$, the optimal set can be written as
\begin{equation}\label{eq:mimc-optimal-set}
 \mathcal I(L) = \en*\{\} {\valpha \in \nset^2 \: : \: (1 - \ssp +
   \ggp)\alpha_1 + (3 - \sst)\alpha_2 \leq L},
\end{equation}
and the optimal
computational complexity of MIMC is
$\Order{\tol^{-2 - 2\max(0,
    \zeta)}\log\op*{\tol^{-1}}^{\mathfrak{p}}}$, where
\[
  \aligned \zeta =& \max\op*{\frac{ \ggp - \ssp - 1}{2}, \frac{
      1-\sst}{2}}, \\
  \xi =& \min\op*{\frac{2 -\ssp}{\ggp}, {2 - \sst
    }} \geq 0, \\
  \mathfrak{p} &=
\begin{cases}
0 & \zeta < 0\\
2 \mathfrak{z} & \zeta = 0 \\
2 (\mathfrak{z} - 1) (\zeta + 1) & \zeta > 0 \text{ and } \xi>0 \\
1 + 2(\mathfrak{z} - 1) (\zeta + 1) & \zeta > 0 \text{ and } \xi=0 \\
\end{cases}\\
  \text{and}\qquad
  \mathfrak{z} &=
\begin{cases}
  1 \hskip 2.7cm & {\ggp-\ssp - 1} \neq {1-\sst} \\
  2 & {\ggp-\ssp - 1} = {1-\sst}.
\end{cases}
& \\
\endaligned
\]
\begin{kuramoto}
  Here again, we use a naive calculation method of $\phi_{P}^N$
  (yielding $\ggp=2$) and the partitioning sampler (yielding
  $\ssp=1$). Finally, using the Milstein time-stepping scheme, we have
  $\sst=2$.  Hence, $\zeta = 0$, $\mathfrak z=1$ and
  \( \work{\mathcal A_{\textnormal{MIMC}}} = \Order{\tol^{-2}
    \log\op*{\tol^{-1}}^2}. \) See also Table~\ref{tbl:rates} for the
  work complexities for different common values of $\sst$ and $\ggp$.
\end{kuramoto}

\begin{table}
  \centering
  \providecommand{\rate}[2]{$(#1, #2)$}
\begin{tabular}{ l|c|c|c|c }
  Method & $\sst=1, \ggp=1$ & $\sst=1, \ggp=2$ & $\sst=2, \ggp=1$ &
                                                                    $\sst=2, \ggp=2$ \\
  \hline
  MC (Section~\ref{ss:mc})  & \rate{3}{0} &  \rate{4}{0}  & \rate{3}{0} & \rate{4}{0} \\
  MLMC (Section~\ref{sec:mlmc-N})  & \rate{2}{2} &  \rate{3}{2}  & \rate{2}{0} & \rate{3}{0} \\
  MLMC (Section~\ref{sec:mlmc-P})  & \rate{3}{0} &  \rate{3}{2}  & \rate{3}{0} & \rate{3}{2} \\
  MLMC (Section~\ref{sec:mlmc-NP}) & \rate{2}{2} &  \rate{3}{0}  & \rate{2}{2} & \rate{3}{0} \\
  MIMC (Section~\ref{ss:mimc})     & \rate{2}{2} &  \rate{2}{4}  &
                                                                   \rate{2}{0} & \rate{2}{2} \\
  \hline
\end{tabular}
\caption{The work complexity of the different methods presented in
  this work in common situations, encoded as $(a,b)$ to represent
  $\Order{\tol^{-a}\p{\log\p{\tol^{-1}}}^b}$. When appropriate, we use
  the partitioning estimator (i.e., $\ssp=1$). In general, MIMC has
  always the best complexity. However, when $\ggp=1$ MIMC does not offer
  an advantage over an appropriate MLMC method.}
  \label{tbl:rates}
\end{table}

\section{Numerical Example}\label{s:res}
In this section we provide numerical evidence of the assumptions and
work complexities that were made in the Section~\ref{s:mc}.  This
section also verifies that the constants of the work complexity (which
were not tracked) are not significant for reasonable error tolerances.
The results in this section were obtained using the \texttt{mimclib}
software library \cite{mimclib} and \texttt{GNU
  parallel}~\cite{Tange2011a}.

In the results outlined below, we focus on the Kuramoto example in
\eqref{eq:kuramoto-system}, with the following choices:
$\sigma = 0.4$, $T=1$, $x_{p}^0 \sim \mathcal N(0, 0.2)$ and
$\vartheta_p \sim \mathcal U(-0.2, 0.2)$ for all $p$.
We also set
\begin{equation}\label{eq:PN-choices}\aligned
  &P_\ell=5 \times 2^\ell &\text{ and } N_\ell=4 \times 2^\ell &\text {
    for MLMC},\\
  \text{and}\qquad &P_{\alpha_1}=5 \times 2^{\alpha_1} &\text{ and }
  N_{\alpha_2}=4 \times 2^{\alpha_2} &\text{ for MIMC}.
  \endaligned
\end{equation}

Figure~\ref{fig:mlmc-rates} shows the absolute expectation and
variance of the level differences for the different MLMC settings that
were outlined in Section~\ref{ss:mlmc}. These figures verify
Assumptions \eqref{eq:P-bias-assumption}, \eqref{eq:P-var-assumption}
and
\eqref{eq:MLMC-N-var-assumption}--\eqref{eq:MLMC-NP-var-assumption}
with the values $\sst=2$ and $\ssp=0$ for the $\overline \varphi$
sampler in \eqref{eq:MLMC-non-antithetic} or the value $\ssp=1$ for
the $\widehat \varphi$ sampler in \eqref{eq:MLMC-antithetic}. For the
same parameter values, Figure~\ref{fig:mimc-rates} provides numerical
evidence for Assumptions~\eqref{eq:MIMC-bias-assumption} and
\eqref{eq:MIP-var-assumption} for the $\widehat \varphi$ sampler
~\eqref{eq:MLMC-antithetic}.

We now compare the MLMC method \cite{giles:MLMC} in the setting that
was presented in Section~\ref{sec:mlmc-NP} and the MIMC method
\cite{abdullatif.etal:MultiIndexMC} that was presented in
Section~\ref{ss:mimc}. In both methods, we use the Milstein
time-stepping scheme and the partitioning sampler, $\widehat \varphi$,
in~\eqref{eq:MLMC-antithetic}. Recall that in this case, we verified
numerically that $\ggp=2$, $\ssp=1$ and $\sst=2$. We also use the MLMC
and MIMC algorithms that were outlined in their original work and use
an initial 25 samples on each level or multi-index to compute a
corresponding variance estimate that is required to compute the
optimal number of samples.  In the following, we refer to these
methods as simply ``MLMC'' and ``MIMC''.  We focus on the settings in
Sections~\ref{sec:mlmc-NP} and~\ref{ss:mimc} since checking the bias
of the estimator in those settings can be done straightforwardly by
checking the absolute value of the level differences in MLMC or the
multi-index differences in MIMC. On the other hand, checking the bias
in the settings outlined in Sections~\ref{ss:mc}, \ref{sec:mlmc-N}
and~\ref{sec:mlmc-P} is not as straightforward and determining the
number of times steps and/or the number of particles to satisfy a
certain error tolerance requires more sophisticated algorithms. This
makes a fair numerical comparison with these later settings somewhat
difficult.

Figure~\ref{fig:error-vs-tol}-\emph{left} shows the exact errors of
both MLMC and MIMC for different prescribed tolerances. This plot
shows that both methods estimate the quantity of interest up to the
same error tolerance; comparing their work complexity is thus fair. On
the other hand, Figure~\ref{fig:error-vs-tol}-\emph{right} is a PP
plot, i.e., a plot of the cumulative distribution function (CDF) of
the MLMC and MIMC estimators, normalized by their variance and shifted
by their mean, versus the CDF of a standard normal distribution. This
figure shows that our assumption in Section~\ref{sec:prob-set} of the
asymptotic normality of these estimators is well founded.
Figure~\ref{fig:maxlvl-vs-tol} shows the maximum discretization level
for both the number of time steps and the number of particles for MLMC
and MIMC (cf. \eqref{eq:PN-choices}). Recall that, for a fixed
tolerance in MIMC, $2 \alpha_2 + \alpha_1$ is bounded by a constant
(cf. \eqref{eq:mimc-optimal-set}).
Hence, Figure~\ref{fig:maxlvl-vs-tol} has a direct implication on the
results reported in Figure~\ref{fig:maxwork-vs-tol} where we plot the
maximum cost of the samples used in both MLMC and MIMC for different
tolerances. This cost represents an indivisible unit of simulation for
both methods, assuming we treat the simulation of the particle system
as a black box. Hence, Figure~\ref{fig:maxwork-vs-tol} shows that MIMC
has %
better parallelization scaling, i.e., even with an infinite number of
computation nodes MIMC would still be more efficient than MLMC.

Finally, we show in Figure~\ref{fig:work-vs-tol} the cost estimates of
MLMC and MIMC for different tolerances.
This
figure clearly shows the performance improvement of MIMC over MLMC and
shows that the complexity rates that we derived in this work are
reasonably accurate. %

\begin{figure}
  \includegraphics[scale=0.48]{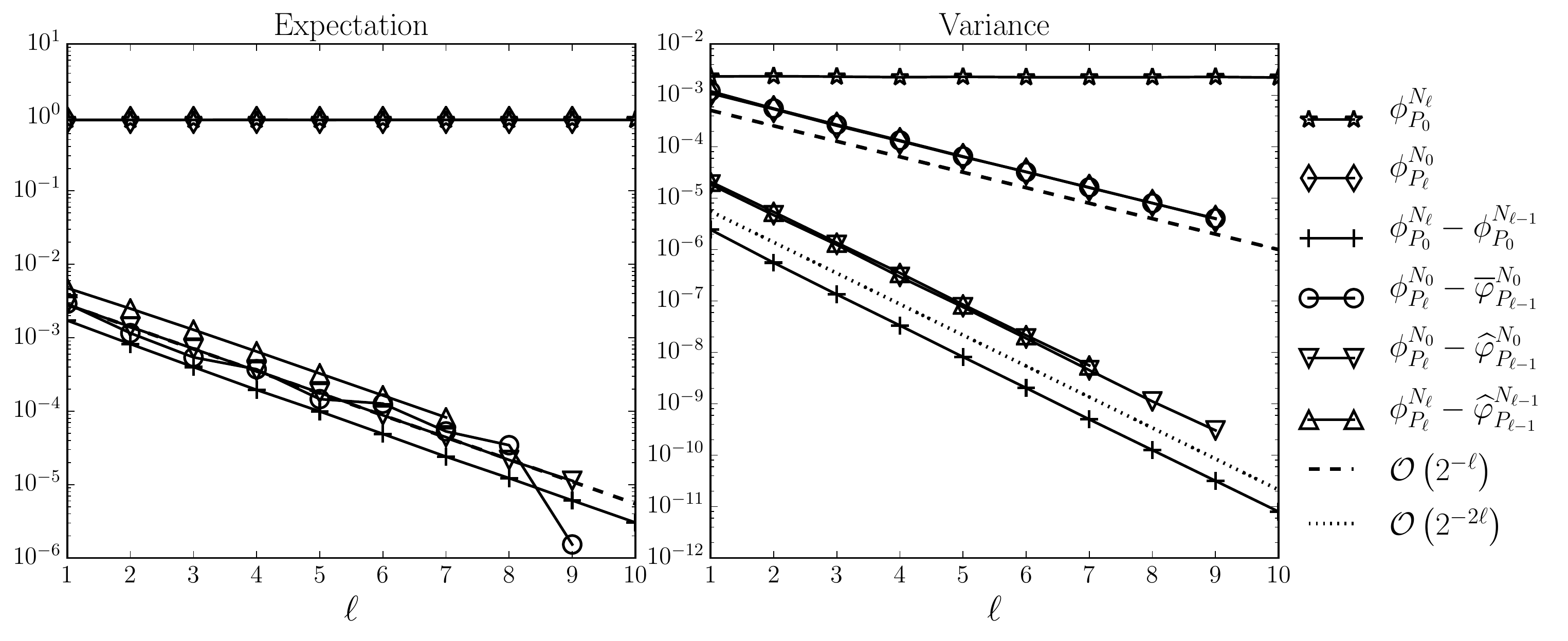}
  \centering
  \caption{ This plot shows, for the Kuramoto example \eqref{eq:kuramoto-system}, numerical
    evidence for Assumption~\eqref{eq:P-bias-assumption} (\emph{left})
    and Assumptions~\eqref{eq:P-var-assumption},
    \eqref{eq:MLMC-N-var-assumption}--\eqref{eq:MLMC-NP-var-assumption}
    (\emph{right}). Here, $P_{\ell}$ and $N_{\ell}$ are chosen
    according to \eqref{eq:PN-choices}. From the \emph{right} plot, we
    can confirm that $\sst=2$ for the Milstein method.  We can also
    deduce that using the $\overline \varphi$ sampler in
    \eqref{eq:MLMC-non-antithetic} yields $\ssp=0$ in
    \eqref{eq:MIP-var-assumption} (i.e., no variance reduction
    compared to $\var{\phi_{P_\ell}^{N_\ell}}$) while using the
    $\widehat \varphi$ sampler in \eqref{eq:MLMC-antithetic} yields
    $\ssp=1$ in \eqref{eq:MIP-var-assumption} (i.e. $\Order{P^{-2}}$).
  }
\label{fig:mlmc-rates}
\end{figure}

\begin{figure}
  \centering
  \includegraphics[scale=0.4, page=8]{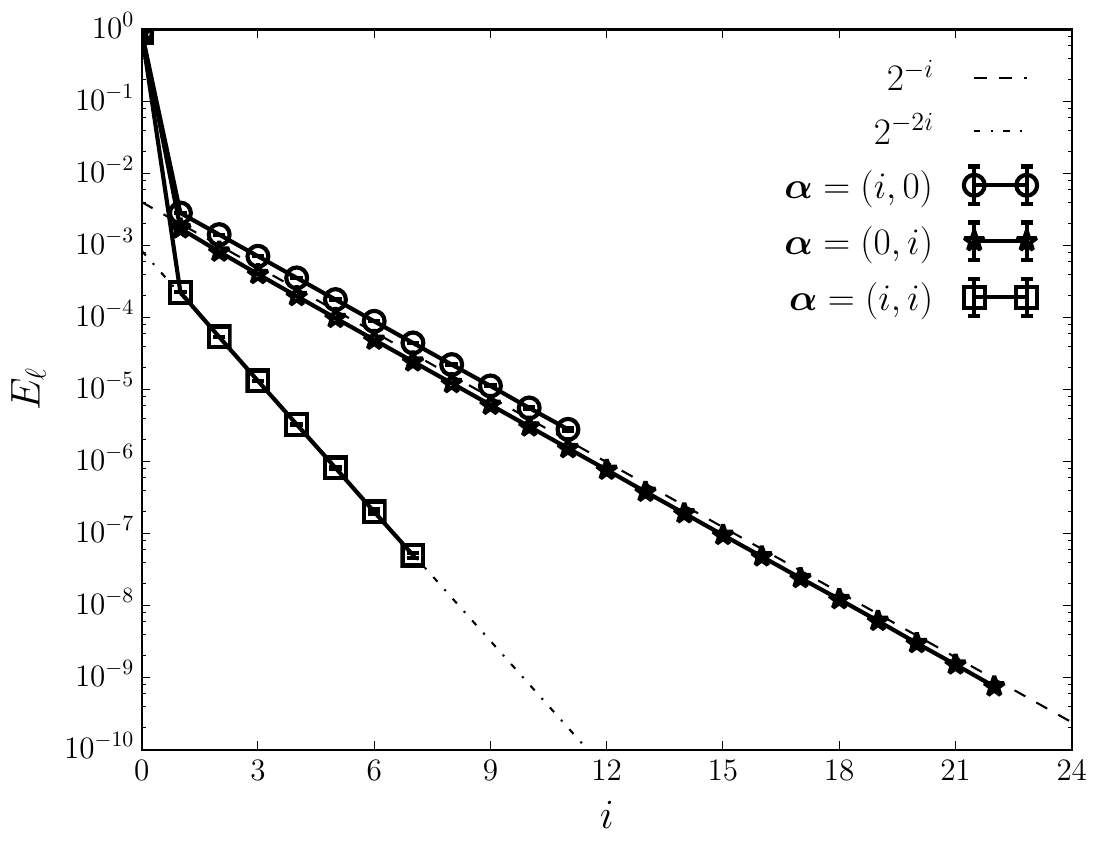}
  \includegraphics[scale=0.4, page=9]{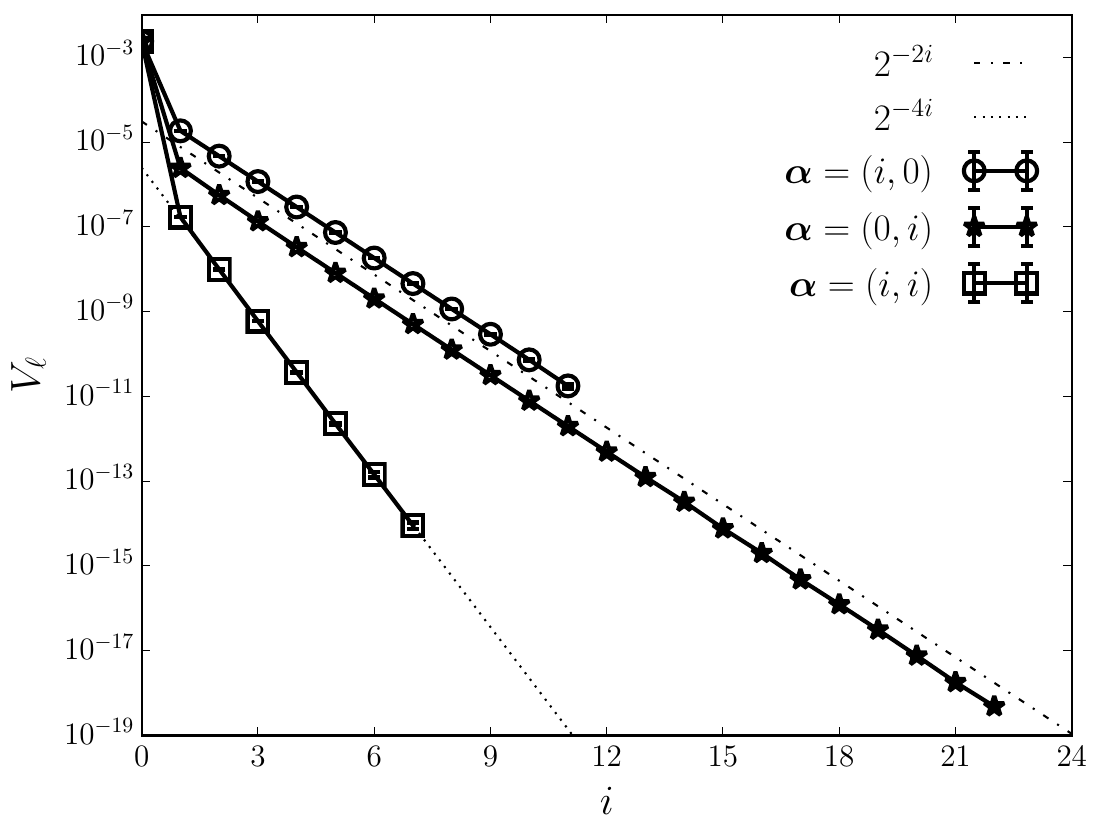}
  \caption{This figure provide numerical evidence
    for~\eqref{eq:MIMC-bias-assumption} ({\em left}) and
    \eqref{eq:MIP-var-assumption} ({\em right}) for the Kuramoto
    example \eqref{eq:kuramoto-system} when using the Milstein scheme
    for time discretization (yielding $\sst=2$) and the partitioning
    sample of particle systems (yielding $\ssp=1$).  Here,
    $P_{\alpha_1}$ and $N_{\alpha_2}$ are chosen according to
    \eqref{eq:PN-choices}. When considering a mixed difference (i.e,
    $\valpha=(i,i)$), a higher rate of convergence is observed.}
\label{fig:mimc-rates}
\end{figure}

\begin{figure}
  \centering
  \includegraphics[scale=0.4, page=1]{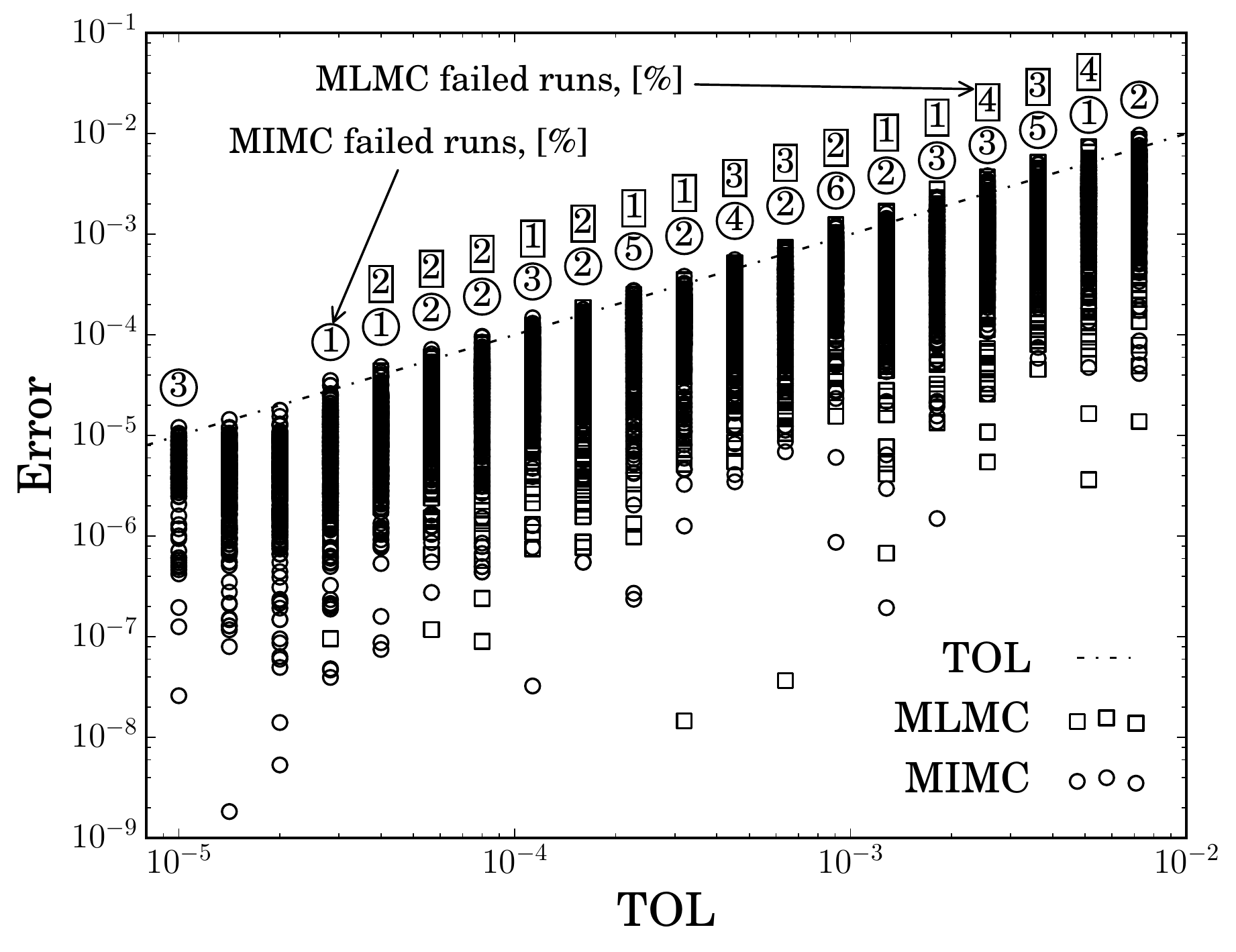}
  \includegraphics[scale=0.4, page=2]{mckean.pdf}
  \caption{In these plots, each marker represents a separate run of
    the MLMC or MIMC estimators (as detailed in
    Sections~\ref{sec:mlmc-NP} and~\ref{ss:mimc}, respectively) when
    applied to the Kuramoto example \eqref{eq:kuramoto-system}.
    \emph{Left}: the exact errors of the estimators, estimated using a
    reference approximation that was computed with a very small
    $\tol$. This plot shows that, up to the prescribed 95\% confidence
    level, both methods approximate the quantity of interest to the
    same required tolerance, $\tol$. The upper and lower numbers above
    the linear line represent the percentage of runs that failed to
    meet the prescribed tolerance, if any, for both MLMC and MIMC,
    respectively \emph{Right}: A PP plot of the CDF of the value of
    both estimators for certain tolerances, shifted by their mean and
    scaled by their standard deviation showing that both estimators,
    when appropriately shifted and scaled, are well approximated by a
    standard normal random variable.}
\label{fig:error-vs-tol}
\end{figure}

\begin{figure}
  \centering
  \includegraphics[scale=0.4, page=3]{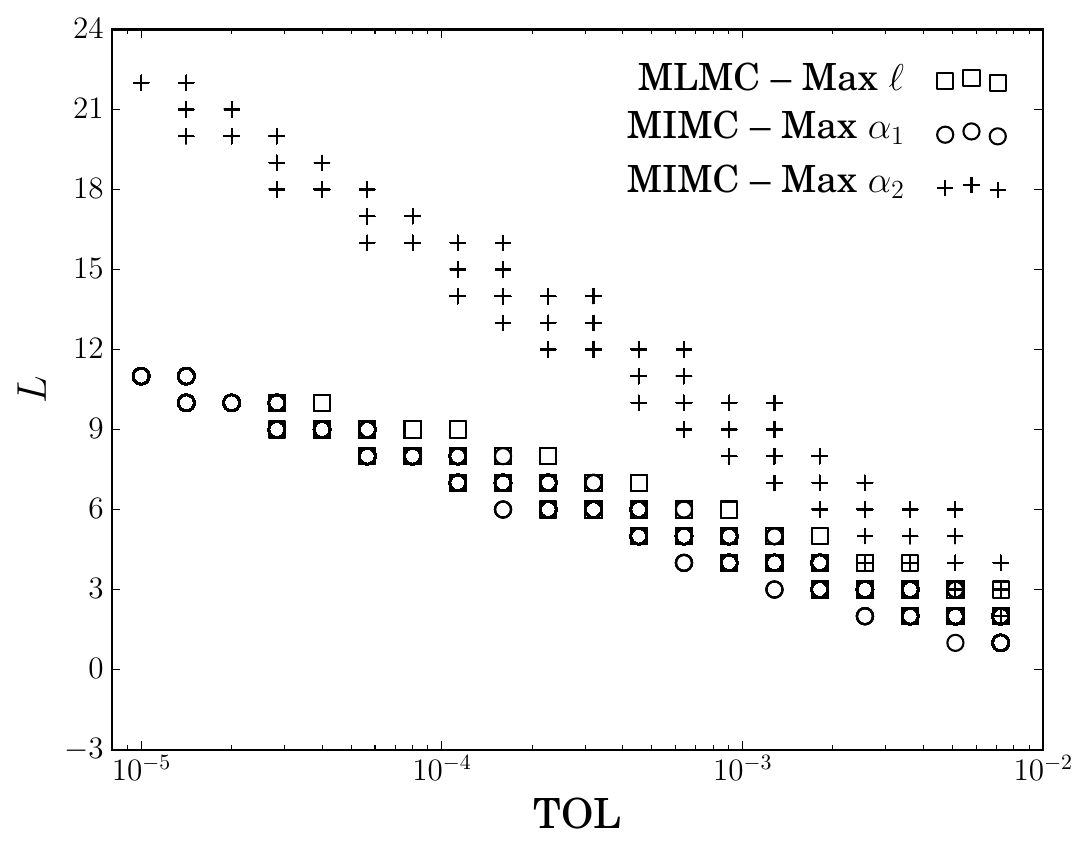}
  \caption{The maximum discretization level of the number of time
    steps and the number of particles for both MLMC and MIMC (as
    detailed in Sections~\ref{sec:mlmc-NP} and~\ref{ss:mimc},
    respectively) for different tolerances
    (cf. \eqref{eq:PN-choices}). Recall that, for a fixed tolerance in
    MIMC, $2 \alpha_2 + \alpha_1$ is bounded by a constant
    (cf. \eqref{eq:mimc-optimal-set}).}
  \label{fig:maxlvl-vs-tol}
\end{figure}

\begin{figure}
  \centering
  \includegraphics[scale=0.4, page=10]{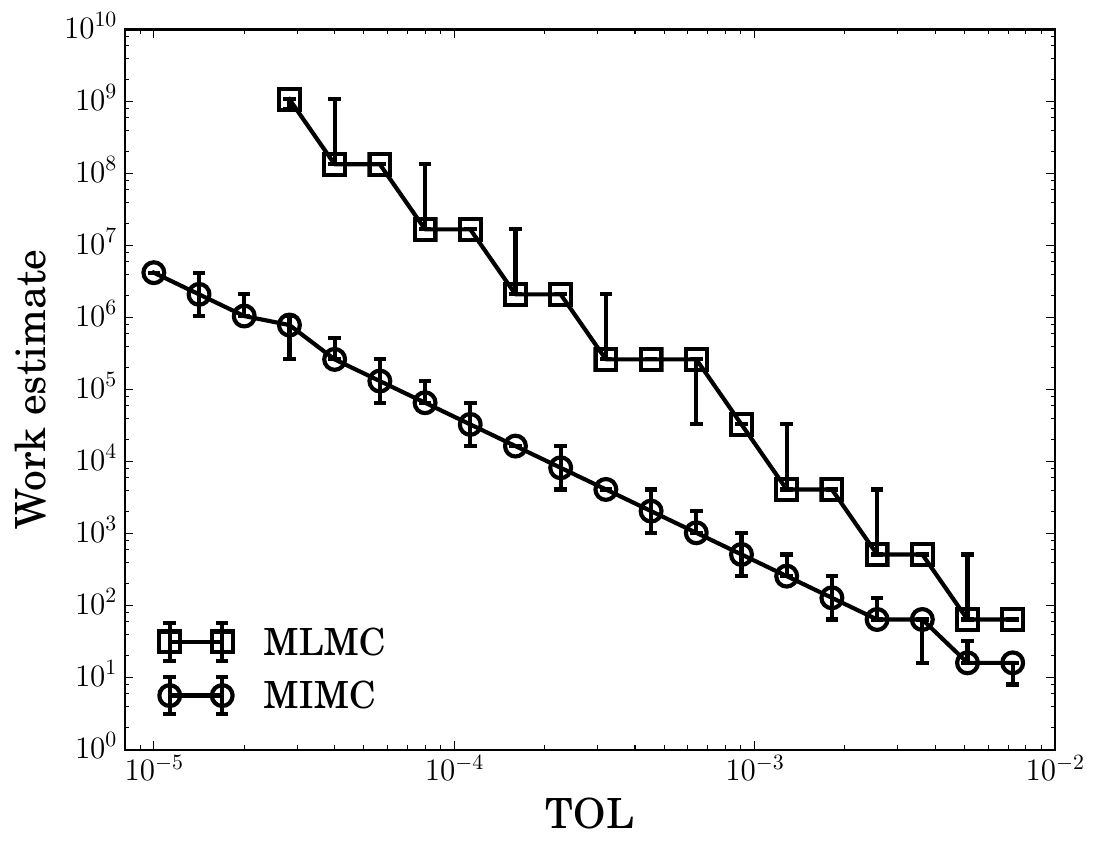}
  \includegraphics[scale=0.4, page=11]{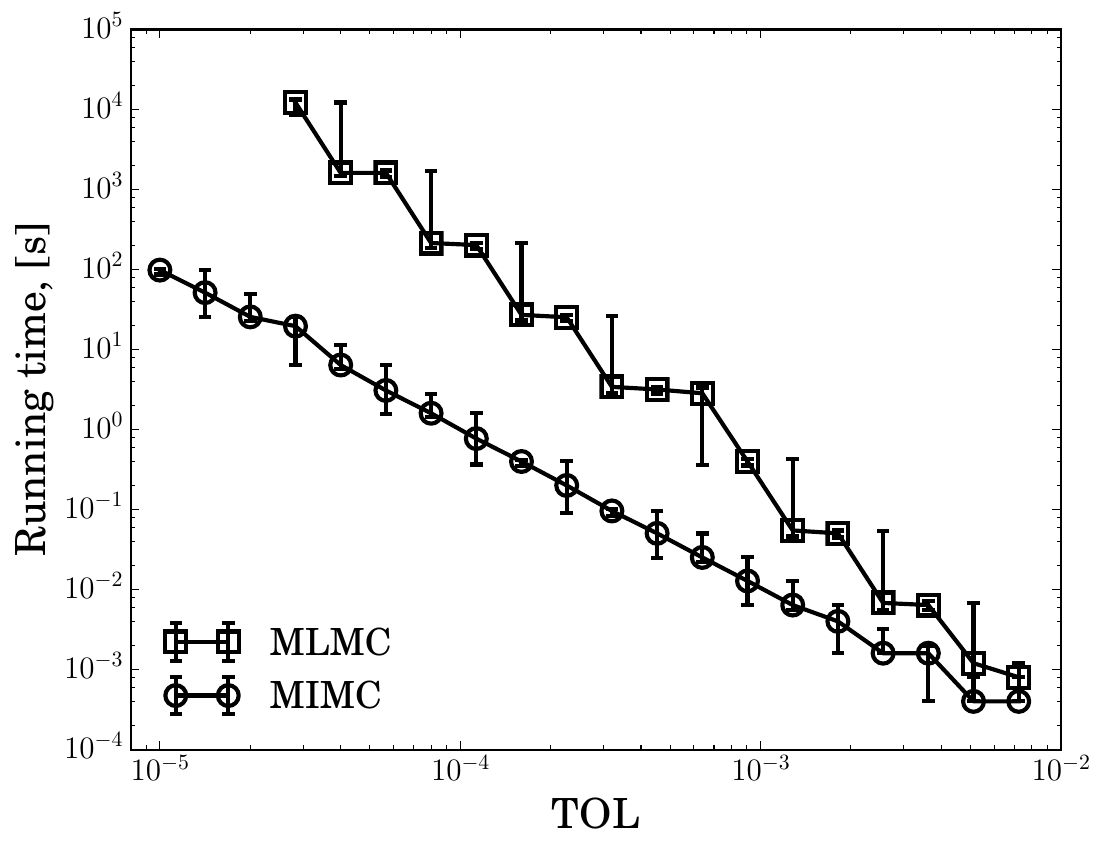}
  \caption{The maximum work estimate ({\em left}) and running time (in
    seconds, {\em right}) of the samples used in MLMC and MIMC (as
    detailed in Sections~\ref{sec:mlmc-NP} and~\ref{ss:mimc},
    respectively) when applied to the Kuramoto example
    \eqref{eq:kuramoto-system}.  These plots show the single
    indivisible work unit in MLMC and MIMC which gives an indication
    of %
    the parallelization scaling of both methods.}
  \label{fig:maxwork-vs-tol}
\end{figure}

\begin{figure}
  \centering
\includegraphics[scale=0.4, page=4]{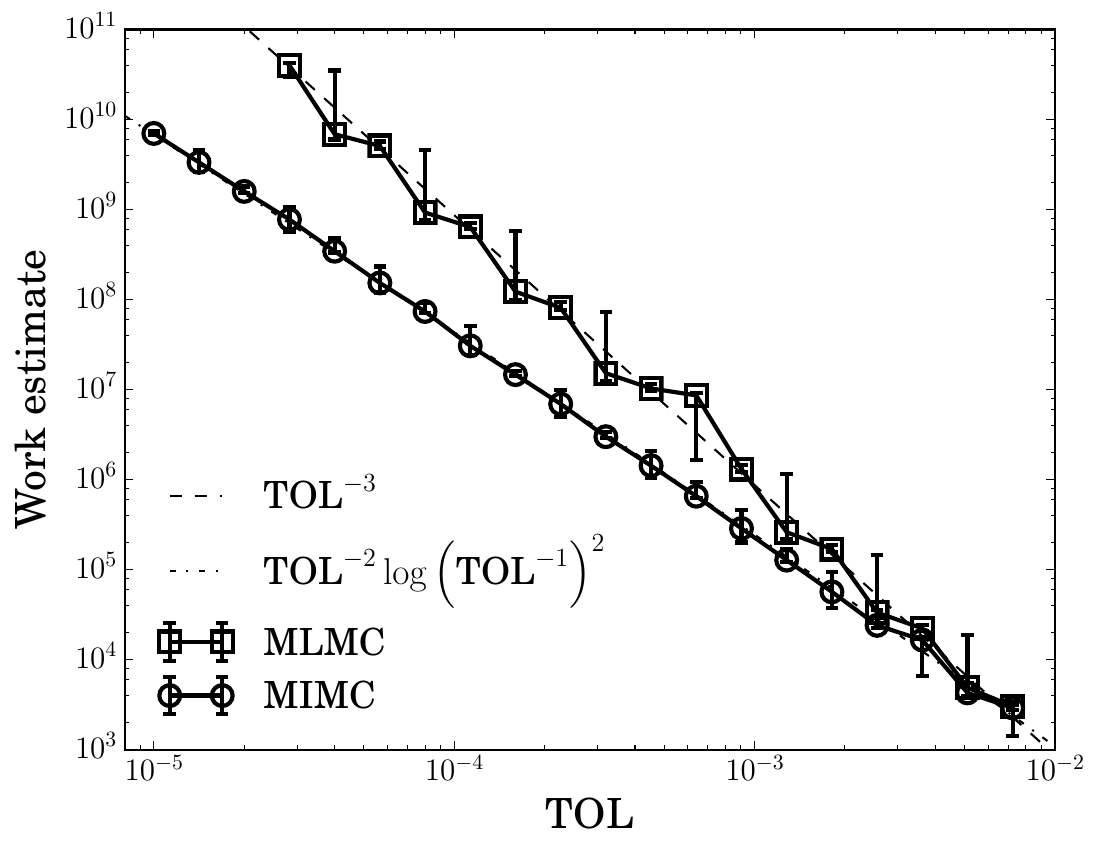}
\includegraphics[scale=0.4, page=5]{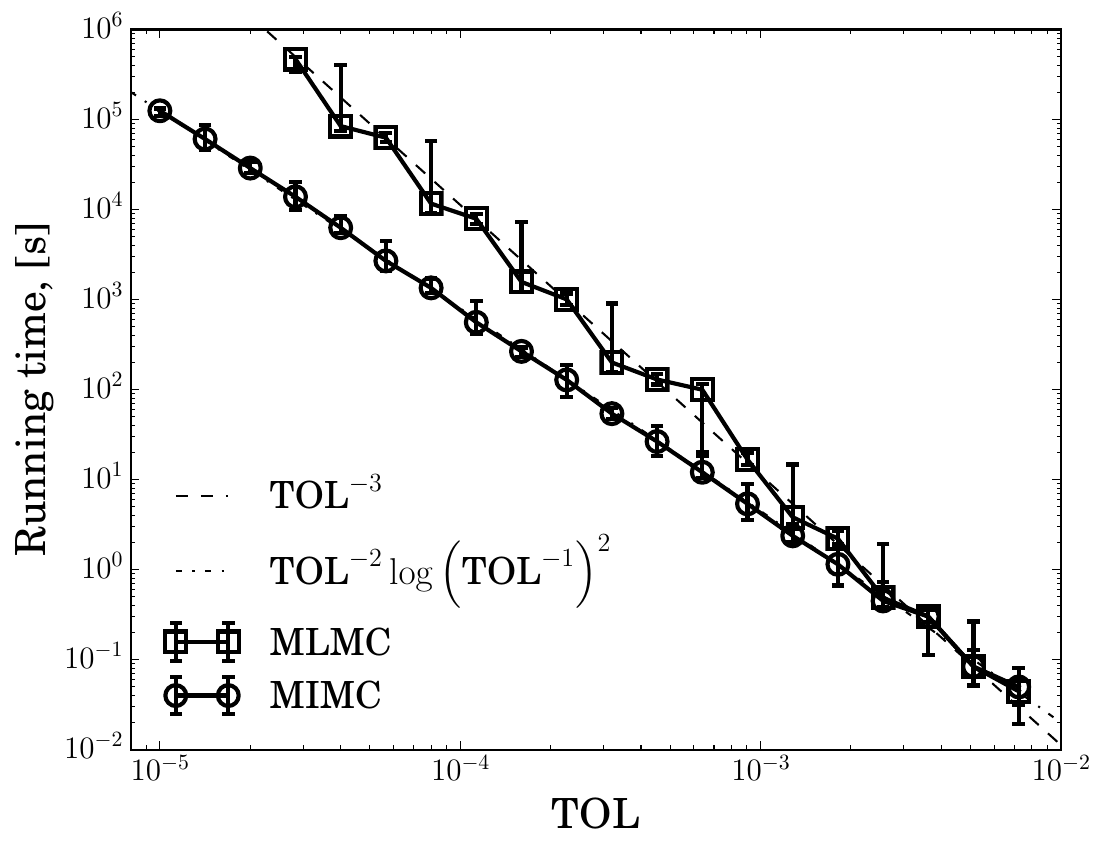}
\caption{Work estimate ({\em left}) and running time (in seconds, {\em
    right}) of MLMC  and MIMC
  (as detailed in
    Sections~\ref{sec:mlmc-NP} and~\ref{ss:mimc}, respectively) when
    applied to the Kuramoto
    example
  \eqref{eq:kuramoto-system}. For sufficiently small tolerances, the
  running time closely follows the predicted theoretical rates (also
  plotted) and shows the performance improvement of MIMC.
}
\label{fig:work-vs-tol}
\end{figure}

 \section{Conclusions}\label{s:conc}
 This work has shown both numerically and theoretically under certain
 assumptions, that could be verified numerically, the improvement of
 MIMC over MLMC when used to approximate a quantity of interest
 computed on a particle system as the number of particles goes to
 infinity. The application to other particle systems (or equivalently
 other McKean-Vlasov SDEs) is straightforward and similar improvements
 are expected. The same machinery was also suggested for approximating
 nested expectations in \cite{giles:acta} and the analysis here
 applies to that setting as well. Moreover, the same machinery, i.e.,
 multi-index structure with respect to time steps and number of
 particles coupled with a partitioning estimator, could be used to
 create control variates to reduce the computational cost of
 approximating quantities of interest on stochastic particle systems
 with a finite number of particles.

 Future work includes analyzing the optimal level separation
 parameters, $\betap$ and $\betat$, and the behavior of the tolerance
 splitting parameter, $\theta$. Another direction could be applying
 the MIMC method to higher-dimensional particle systems such as the
 crowd model in \cite{hajiali:msthesis}. On the theoretical side, the
 next step is to prove the assumptions that were postulated and
 verified numerically in this work for certain classes of particle
 systems, namely:
 the second order convergence with respect to the number of particles
 of the variance of the partitioning estimator
 \eqref{eq:MLMC-antithetic} and the convergence rates for mixed
 differences~\eqref{eq:MIMC-bias-assumption} and~
 \eqref{eq:MIP-var-assumption}.

 \section*{Acknowledgments}
R. Tempone is a member of the KAUST Strategic Research Initiative,
Center for Uncertainty Quantification in Computational Sciences and
Engineering. R. Tempone received support from the KAUST CRG3 Award
Ref: 2281 and the KAUST CRG4 Award Ref:2584.

The authors would like to thank Lukas Szpruch for the valuable
discussions regarding the theoretical foundations of the methods.
\bibliographystyle{acm}
 \end{document}